\documentclass[12pt, a4paper, oneside]{amsart}
\usepackage[a4paper, total={6in, 10in}]{geometry}
\usepackage[T1]{fontenc}
\usepackage[utf8]{inputenc}

\usepackage{todonotes}

\usepackage{amssymb}   
\usepackage{amsthm}      
\usepackage{thmtools}     
\usepackage{mathtools}   
\usepackage{mathrsfs}     
\usepackage{physics}       

\usepackage[backend=biber,style=numeric-comp, sorting=nyt, maxnames=99]{biblatex}
\usepackage[hidelinks, hypertexnames=false]{hyperref}
\usepackage[nameinlink, noabbrev, capitalize]{cleveref}
\usepackage{bookmark}

\declaretheorem[style = plain, numberwithin=section]{theorem}
\declaretheorem[style = plain,      sibling = theorem]{corollary}
\declaretheorem[style = plain,      sibling = theorem]{lemma}
\declaretheorem[style = plain,      sibling = theorem]{proposition}
\declaretheorem[style = definition, sibling = theorem]{definition}
\declaretheorem[style = definition, sibling = theorem, qed=$\spadesuit$]{example}
\declaretheorem[style = remark,    numbered = no]{remark}

\declaretheorem[style = remark,    numbered = no]{editorial comment}

\usepackage[scaled]{beramono}  
\usepackage{listings}
\usepackage{xcolor}

\definecolor{codegreen}{rgb}{0,0.6,0}
\definecolor{codegray}{rgb}{0.5,0.5,0.5}
\definecolor{backcolour}{rgb}{0.95,0.95,0.92}
\definecolor{codegreen}{rgb}{0,  0.6,  0.35}

\lstdefinestyle{mystyle}{
	backgroundcolor=\color{backcolour},   
	commentstyle=\color{codegreen},
	keywordstyle=\color{blue},
	numberstyle=\tiny\color{codegray},
	stringstyle=\color{codegreen},
	basicstyle=\ttfamily\footnotesize,
	breakatwhitespace=false,         
	breaklines=true,                 
	captionpos=b,                    
	keepspaces=true,                 
	numbers=left,                    
	numbersep=5pt,                  
	showspaces=false,                
	showstringspaces=false,
	showtabs=false,                  
	tabsize=2
}
\usepackage[shortlabels]{enumitem}
\setenumerate{label=(\roman{enumi})}

\newcommand{\N}{\mathbb{N}}   
\newcommand{\Z}{\mathbb{Z}}   
\newcommand{\Q}{\mathcal{Q}}   
\newcommand{\R}{\mathbb{R}}   
\newcommand{\C}{\mathbb{C}}   
\newcommand{\bbA}{\mathbb{A}}   
\newcommand{\A}{\mathscr{A}}
\newcommand{\B}{\mathscr{B}}

\renewcommand{\P}{\mathcal{P}} 

\newcommand{\U}{\mathcal{U}}
\newcommand{\I}{\mathscr{I}}
\newcommand{\J}{\mathscr{J}}
\newcommand{\K}{\mathscr{K}}
\newcommand{\T}{\mathbb{T}}
\newcommand{\sC}{\mathscr{C}}

\newcommand{\gbd}{\propto_\Gamma}

\usepackage{graphicx}                                                         

\usepackage[all]{xy}

\title{On AF- and type I-ideals in certain crossed product C*-algebras}

\author{Alexander G. Ravnanger}
\address{ \noindent \newline 
	Department of Mathematical Sciences \newline 
	University of Copenhagen \newline
	Universitetsparken 5 \newline 
	2100 Copenhagen}
\email[]{agr@math.ku.dk}

\begin{document}
	\begin{abstract}
		We study locally finite-dimensional ideals in crossed products of totally disconnected spaces by free actions of the integers and in uniform Roe algebras of exact discrete groups. In the first case, we present a dynamical description of the largest locally finite-dimensional ideal, which turns out to coincide with the intersection of all maximal ideals. In the latter case, we provide a coarse geometric characterization of the locally finite-dimensional compact ideals. Moreover, we show that for crossed products of totally disconnected spaces by free actions of exact groups, the largest type I-ideal is locally finite-dimensional. In the case of uniform Roe algebras, we provide coarse geometric conditions for compact ideals guaranteeing that the ideal is type I and admits an embedding of a UHF-algebra, respectively.
	\end{abstract}
	
	\maketitle

	\section{Introduction}
	
	 Since its first systematic study by Zeller-Meier in \cite{zeller}, the crossed product construction has played an instrumental role as a source of examples of interesting C*-algebras. Hence, it is of keen interest to understand how C*-algebraic regularity properties of the crossed product arise in terms of the underlying dynamics. Significant effort has been spent towards the case of minimal actions in this line of research. This is in part because minimal systems are particularly well-behaved, but also because minimality, under the additional hypothesis of topological freeness, gives rise to simple C*-algebras, which make up the building blocks of the theory and have been the focus of the classification program for C*-algebras. Early examples of this include Putnam's work on the structure of crossed products associated to Cantor minimal systems in \cite{putnam} and \cite{putnam_AF}, and the studies of the irrational rotation algebras, e.g., \cite{rieffel_irr,blackadar_irr,pv_irr,elliott_irr}. However, these questions are also of huge independent interest in the non-simple setting, which gives rise to new and interesting phenomena, while at the same time presenting new challenges, and indeed regularity properties for non-minimal dynamics are gaining more attention at the moment. For example, Cantor systems with finitely many minimal components were systematically treated in \cite{bezuglyi}, and non-minimal free actions on the circle were studied in the recent preprint \cite{bell}. 
	 
	 While uniform Roe algebras were originally introduced by Roe in his work on index theory, they have since become popular to study as interesting examples of C*-algebras, that also provide an interesting bridge between coarse geometry and operator algebras. Unlike what is arguably the most popular set-up for studying regularity properties of crossed products, uniform Roe algebras are non-separable and far from simple. Much of the literature on uniform Roe algebras analyzes them from the perspective of coarse geometry. However, in the case of groups, much can be gained from studying them as crossed products, see for example \cite{kellerhals}. The recent extensive preprint \cite{braga2026} takes this further and studies the large scale (irreversible) dynamics of metric spaces. Until recently, the rigidity question for uniform Roe algebras was arguably the most important problem in the subject. At the moment, one of the central issues is to compute their nuclear dimension, see \cite[Problem LXXXVII]{99prob}. The case of nuclear dimension zero, i.e., local finite-dimensionality, was settled by Li and Willett in \cite{li2018low}, where they show that a uniform Roe algebra of uniformly locally finite space is locally finite-dimensional if and only if the underlying space has asymptotic dimension zero in the sense of Gromov. 
	 
	 A highly relevant line of inquiry that has enjoyed significant interest for a long time, is to understand the ideal structure of non-simple crossed products. Many results are known in this direction, catering to different specific situations, but in general it is now well-established that exactness combined with some freeness condition on the action is the appropriate setting in which to expect the \textit{ideal separation property}, i.e., that every ideal in the crossed product comes from an invariant ideal in the C*-algebra that is being acted on. Many interesting regularity properties have the feature that every C*-algebra admits a largest (possibly trivial) ideal with that given property. Examples include local finite-dimensionality, postliminality and nuclearity. This raises an interesting question for a crossed product enjoying the ideal separation property: Given a regularity property \(\P\), what invariant ideal of the C*-algebra induces the largest ideal in the crossed product with property \(\P\)? In this paper, we address this question for local finite-dimensionality and postliminality, focusing on crossed products of totally disconnected spaces by the integers and uniform Roe algebras of exact discrete groups. In the first case, we provide a dynamical description of the largest locally finite-dimensional ideal, building on the works of Poon and Putnam in \cite{poon} and \cite{putnam}. We show that the largest locally finite-dimensional ideal is AF and that it coincides with the intersection of all maximal ideals. Motivated by the relationship between isolated points and finite-dimensional representations of projections, we also consider the largest type I-ideal, which is shown to be AF. For uniform Roe algebras, we extend the main theorem of \cite{li2018low} by providing a coarse geometric description of the locally finite-dimensional compact ideals. This is succeeded by a discussion of when these compact ideals are type I, providing necessary coarse geometric conditions for the ideal to be type I and admit an embedding of a UHF-algebra, respectively. Although we did not succeed in giving a complete characterization of postliminality, we conclude by showing that in the uniform Roe algebra of the integers, the largest ideal of type I is strictly contained in the largest AF-ideal. 

	This paper has an elliptic composition in that it revolves around two focal examples, namely crossed products of totally disconnected spaces by free actions of the integers and uniform Roe algebras of exact discrete groups, in that order. The two parts are unified by the uniform Roe algebra of the integers as a motivating example. 
	
	\subsection*{Acknowledgements} I am grateful to my supervisor Mikael Rørdam for invaluable discussions and suggestions during the work presented here, as well as for reading early drafts of this paper. Some of this work was also done while I was visiting the Mathematical Institute at the University of Oxford during the spring of 2026. I am grateful to the institute for its hospitality and to Mehrdad Kalantar and Stuart White for hosting me. Thanks to Stuart White for several helpful discussions about the topics of this paper while I was there. The set presented in \cref{ex_typeI_notAF} was constructed during an interaction with Google's Gemini 3.6 Flash. Anthropic's Claude Opus 5 has been used to find typos. 
	
	\section{Preliminaries} 
	
	\subsection{AF- and LF-algebras}
	\label{subsec_AF}
	
	A C*-algebra is called \textit{AF} \textit{(approximately finite- \\dimensional)} if it is the direct limit of a (not necessarily sequential) net of finite-dimensional C*-algebras, indexed over a directed set. By a directed set, we mean a preordered set where every pair of elements has a common upper bound. Equivalently, a C*-algebra is AF if it admits an increasing net of finite-dimensional C*-subalgebras with dense union. A C*-algebra \(\A\) is called \textit{LF (locally finite-dimensional)} if for every finite set of elements \(a_1, \ldots, a_n\) in \(\A\) and \(\varepsilon > 0\), there is a finite-dimensional C*-subalgebra \(\B \subseteq \A\) containing elements \(b_1, \ldots,  b_n \in \B\) such that \(\norm{a_k-b_k} < \varepsilon\) for \(k = 1, \ldots, n\). Clearly, an AF-algebra is LF, and in his foundational work \cite {bratteli_inductive} on the subject, Bratteli showed that for separable C*-algebras the converse is also true. Farah and Katsura gave examples of LF-algebras that are not AF in \cite{farah}. It is well-known that sequential direct limits of separable AF-algebras are AF and that extensions of separable AF-algebras remain AF. However, neither statements hold in general, see \cite{farah} and \cite{bice2019c}. We record the following for completeness. 
	
	\begin{proposition}
		\label{prop_LF}
		Every C*-algebra admits a largest LF ideal. 
	\end{proposition}
	
	\begin{proof}
		It follows from the analogous fact in the separable setting that extensions of LF-algebras remain LF, see \cite[Theorem 1.13]{bice2019c}. If \(\I\) and \(\J\) are two LF-ideals in a C*-algebra \(\A\), we claim that \(\I+\J\) is LF. Indeed, we have \((\I+\J)/\J \cong \I/\I \cap \J\) is LF, and thus \(\I+\J\) is an extension of two LF-algebras, hence LF by the comment above. Thus, the family of LF ideals of \(\A\) is directed by inclusion, and the closed union of all LF ideals is the largest LF ideal in \(\A\). 
	\end{proof}

	\subsection{The Poon-Putnam constructions}
	
	Given an action of a group \(\Gamma\) on a C*-algebra \(\A\), we denote by \(\A \rtimes_r \Gamma\) the associated reduced crossed product, and we refer to elements of \(\A\) viewed inside \(\A \rtimes_r \Gamma\) as \textit{diagonal}. A C*-algebra is called an A\(\T\)-algebra if it is the direct limit of a net of C*-algebras that are direct sums of matrix algebras over \(C(\T)\). In \cite{putnam}, Putnam showed that crossed products of Cantor minimal systems are A\(\T\)-algebras. This is important for us because it applies to the simple quotients of our crossed products of interest. By standard approximation techniques, the result may be applied in the non-separable setting. 
	
	Throughout this paper, we shall by a (topological) space always mean a Hausdorff space. 
	
	\begin{lemma}
		\label{lemma_approx}
		Suppose \(\Gamma\) is a discrete countable group acting minimally (resp. freely) on a totally disconnected space \(X\). Given any \(a_1, \ldots, a_n \in C(X) \rtimes_r \Gamma\), there is a separable invariant C*-subalgebra \(\A \subseteq C(X)\) such that \(\widehat{\A}\) is totally disconnected, the action of \(\Gamma\) on \(\widehat{\A}\) is minimal (resp. free )and \(a_1, \ldots, a_n \in \A \rtimes_r \Gamma\). 
	\end{lemma}
	
	\begin{proof}
		Since the group is countable, it is straightforward to see that any finite subset of \(C(X) \rtimes_r \Gamma\) is contained in the crossed product by \(\Gamma\) of some separable invariant C*-subalgebra of \(C(X)\). Combining this with \cite[Proposition 2.2]{blackadar_weak} and \cite[Lemma 2.5.2]{sierakowski_thesis}, yields the statement. 
	\end{proof}
	
	\begin{theorem}[Putnam]
		\label{thm_putnam}
		Let \(X\) be an infinite compact totally disconnected space and \(\varphi \colon X \to X\) a minimal homeomorphism of \(X\). Then the crossed product \(C(X) \rtimes  \Z\) is a simple locally A\(\T\) algebra of real rank zero. 
	\end{theorem}
	
	\begin{proof}
		Putnam showed in \cite[Theorem 2.1]{putnam} that \(C(X) \rtimes  \Z\) is a locally A\(\T\) algebra. If \(X\) is metrizable, then \(C(X) \rtimes  \Z\) is separable, and we conclude from \cite[§ 4.3]{elliott} that \(C(X) \rtimes  \Z\) is an A\(\T\)-algebra. Since projections separate traces on \(C(X) \rtimes  \Z\), we conclude from \cite[]{blackadar_etal} that \(C(X) \rtimes  \Z\) has real rank zero. Combining this with \cref{lemma_approx}, the statement follows. 
	\end{proof}
	
	\begin{remark}
		If \(X\) is finite, then \(C(X) \rtimes \Z \cong C(\T) \otimes M_n(\C)\), where \(n\) is the cardinality of \(X\), for any minimal homeomorphism \(\varphi \colon X \to X\). 
	\end{remark}
	
	Consider a homeomorphism \(\varphi \colon X \to X\) of a totally disconnected space \(X\). For a closed subset \(Z \subseteq X\), denote by \(\A_Z\) the C*-subalgebra of \(C(X) \rtimes  \Z\) generated by \(C(X)\) and \(u C_0(X \setminus Z)\). In \cite[Theorem 3.3]{putnam_AF}, Putnam showed that if \(\varphi\) is minimal and \(X\) is metrizable, then \(\A_{\{z\}}\) is AF for any \(z \in X\). In \cite[Theorem 2.3]{poon}, Poon formulated a more general dynamical condition on \(\varphi\) ensuring that this construction works. We simply observe that metrizability is irrelevant to the proof. 
	
	\begin{theorem}[Poon]
		\label{thm_poon}
		Let \(\varphi\) be a homeomorphism of a totally disconnected space \(X\), and let \(Z \subseteq X\) be a closed subset. Then the C*-algebra \(\A_Z\) is AF if and only if for every clopen neighborhood \(W\) of \(Z\), we have \[\bigcup_{k \in \Z} \varphi^k(W) = X.\]
	\end{theorem}
	
	\begin{proof}
		Necessity is proven exactly like in Poon's paper, see also the proof of \cref{thm_poon_deluxe}. For necessity, Poon shows in \cite[Lemma 2.2]{poon} that whenever \(W\) is a clopen set satisfying the condition above, then \(\A_W\) is finite-dimensional. If \(Z\) is closed, the set of clopen subsets containing it is directed by reverse inclusion since the intersection of two such sets is still a clopen set containing \(Z\). The corresponding C*-subalgebras form an increasing directed system of C*-subalgebras of \(C(X) \rtimes  \Z\) whose union is dense in \(\A_Z\). 
	\end{proof}
	
	\subsection{Uniform Roe algebras and coarse geometry}
	
	For a metric space \(X\), a point \(x \in X\) and \(r > 0\), we denote by \(B_X(x, r)\) the open ball in \(X\) of radius \(r\) around \(x\). We say that a metric space \(X\) is \textit{uniformly locally finite} if for any \(r > 0\), we have \[\sup_{x \in X} \abs{B_X(x, r)} < \infty,\] where \(\abs{\cdot} \) denotes cardinality. A uniformly locally finite space is often referred to as a metric space with \textit{bounded geometry}. Notice that a uniformly locally finite space is necessarily proper, discrete and countable. We will mainly be interested in countable discrete groups and subsets of such. A countable discrete group always admits a proper right (or left) invariant metric, which is unique up to coarse equivalence, see \cite[Proposition 5.5.2]{brown_ozawa}. 
	
	Let \(X\) be a uniformly locally finite space with metric \(d\), and consider the Hilbert space \(\ell^2(X)\) with the standard basis of Dirac functions \((\delta_x)_{x \in X}\). The propagation of a linear operator \(a\) on \(\ell^2(X)\) is \[\operatorname{prop}(a) = \sup \{ d(x, y) : \langle a \delta_y, \delta_x \rangle \neq 0\}.\] Since \(X\) is uniformly locally finite, any linear operator with finite propagation and uniformly bounded matrix coefficients is bounded, see \cite[Lemma 8.1]{winter}. The \textit{uniform Roe algebra} of \(X\), denoted by \(C_u^*(X)\), is the norm-closure of the algebra of such operators in \(\B(\ell^2(X))\). If \(X\) is a discrete group \(\Gamma\), we will always consider it with the unique coarse structure described above. In that case, the uniform Roe algebra may be defined as the crossed product \(\ell^\infty(\Gamma) \rtimes_r \Gamma\) by the left-translation action of \(\Gamma\) on \(\ell^\infty(\Gamma)\). It is not hard to see that then \(C_u^*(\Gamma)\) is the C*-algebra generated by \(\ell^\infty(\Gamma)\) and the reduced group C*-algebra \(C_r^*(\Gamma)\), see \cite[Proposition 5.1.3]{brown_ozawa}. 
	
	The spectrum of \(\ell^\infty(\Gamma)\) is the \textit{Stone--\v{C}ech compactification} \(\beta \Gamma\) of \(\Gamma\). The projections in the C*-algebra \(\ell^\infty(\Gamma)\) are indicator functions of subsets of \(\Gamma\), and they span a dense subalgebra. For a subset \(A \subseteq \Gamma\), the projection \(1_A\) also corresponds to a clopen (closed and open) subset \(K_A\) of \(\beta \Gamma\), which is the closure of \(A\) in \(\beta \Gamma\). As a point set, \(\beta \Gamma\) consists of ultrafilters on \(\Gamma\), and the set \(K_A\) is the set of ultrafilters containing \(A\). The Stone--\v{C}ech compactification is the universal compactification of \(\Gamma\) as a (discrete) topological space, in the sense that any function \(f \colon \Gamma \to K\) into a compact space extends uniquely to a continuous function \(\beta f \colon \beta \Gamma \to K\). It follows from this universal property that the group operation on \(\Gamma\) extends to an associative operation on \(\beta \Gamma\), turning it into a compact right topological semigroup. For more about the algebra of \(\beta \Gamma\), see \cite{hindman_strauss_alg}. We may also view the extension of the left multiplication as an action of \(\Gamma\) on \(\beta \Gamma\). This action is always free, see \cite[Theorem 3.34]{hindman_strauss_alg}, and it is amenable if and only if \(\Gamma\) is an exact group, see \cite[Theorem 5.1.7]{brown_ozawa}. The action of a discrete group on its Stone--\v{C}ech compactification is the universal free action of \(\Gamma\). Moreover, all of its minimal components are isomorphic to the universal minimal free action of \(\Gamma\). 
	
	We will work with uniform Roe algebras both from the dynamical and coarse geometry perspectives. For a subset \(A \subseteq \Gamma\), we denote by \(1_A\) the projection associated to \(A\) both when working with \(C_u^*(\Gamma)\) in the standard representation on \(\ell^2(\Gamma)\) and when treating it as a crossed product. In other words, we somewhat abusevily also denote \(1_{K_A} \in C(\beta \Gamma)\) by \(1_A\). 
	
	There is a coarse analogue of the Lebesgue covering dimension for topological spaces, due to Gromov, called \textit{asymptotic dimension}. We recall the definition and provide another characterization of the base case, which is useful for our purposes. 
	
	\begin{definition}
		Let \(X\) be a metric space. For a non-negative integer \(d\), we say that \(X\) has \textit{asymptotic dimension at most} \(d\) if for every \(r > 0\), there is a partition \(\U\) of \(X\) into subsets of \(X\) with a further partition \[\U = \U^0 \cup \U^1 \cup \cdots \cup \U^d\] such that for each \(i = 0, 1, \ldots, d\), the collection \(\U^{i}\) is uniformly bounded and \(r\)-separated, i.e., 
		\[\sup \{ \operatorname{diam}(U) : U \in \U\} < \infty \quad \text{and} \quad d(U, V) > r \text{ for all distinct } U, V \in \U^{i}.\]
		The asymptotic dimension of \(X\) is the least \(d\) such that \(X\) has asymptotic dimension at most \(d\). 
	\end{definition}
	
	The following is well-known and straightforward. It was also used in \cite{li2018low} in the proof that uniform Roe algebras are AF if and only if the underlying metric space has asymptotic dimension zero. 
	\begin{proposition}
		\label{prop_equiv_rel}
		Let \(X\) be a uniformly locally finite space. For \(r > 0\), let \(\sim_r\) be the equivalence relation on \(X\) defined by \(x \sim_r y\) if and only if there is a finite sequence of points \(x = x_0, x_1, \ldots, x_n = y\) in \(X\) such that \(d(x_{i+1}, x_i) \leq r\) for \(i = 0, 1, \ldots, n-1\). The following are equivalent:
		\begin{enumerate}
			\item The space \(X\) has asymptotic dimension zero. 
			\item For every \(r > 0\), the equivalence relation \(\sim_r\) has uniformly bounded equivalence classes. 
			\item For every \(r > 0\), the equivalence relation \(\sim_r\) has uniformly finite equivalence classes. 
		\end{enumerate}
	\end{proposition}
	
	It was shown in \cite[Theorem 2]{smith} that a countable group has asymptotic dimension zero if and only if it is \textit{locally finite}, i.e., every finitely generated subgroup is finite. 
	
	\section{Crossed products by the integers}
	
	\subsection{\(K\)-theory}
	
	An obvious obstruction for a C*-algebra to be AF is having non-trivial \(K_1\). As we will see, this is the only obstruction in the case of crossed products of totally disconnected spaces by free, topologically transitive actions of the integers. Recall that an action of a group \(\Gamma\) on a space \(X\) is called \textit{topologically transitive} if for any pair \(U, V\) of non-empty open subsets of \(X\), there is an element \(g \in \Gamma\) such that \(U \cap gV \neq \emptyset\). The following computation is standard, but since it lies at the heart of the proof of the main theorem in this section, we include the proof. 
	
	Consider a totally disconnected compact Hausdorff space \(X\) and a homeomorphism \(\varphi \colon X \to X\). For an invariant open set \(U \subseteq X\), set \[B(U, \Z) = \{ f - f \circ \varphi^{-1} : f \in C_c(U, \Z) \}.\] The \(0\)th \textit{homology} of \((U, \varphi)\) is \[H_0(U, \varphi) = C_0(U, \Z) / B(U, \Z).\] We say that an action is \textit{indecomposable} if it admits no non-trivial invariant clopen sets. A homeomorphism of a space is called \(\textit{aperiodic}\) if it has no finite orbits. A homeomorphism is aperiodic if and only if the corresponding action by \(\Z\) is free. 
	
	\begin{proposition}
		\label{prop_K_ideals}
		Let \(X\) be a locally compact space and suppose that \(\varphi \colon X \to X\) is aperiodic, indecomposable and topologically transitive. The \(K\)-theory of an ideal \(\I \subseteq C_0(X) \rtimes \Z\) is given by  
		\[K_0(\I) = H_0(U, \varphi) \quad \text{and} \quad K_1(\I) \cong \begin{cases}
			\Z, \quad & \text{if } \I = C(X) \rtimes \Z \\
			0, \quad & \text{if } \I \text{ is proper},
		\end{cases}\]
		where \(U \subseteq X\) is the invariant open set such that \(\I \cap C_0(X) = C_0(U)\).
	\end{proposition}
	
	\begin{proof} 
		Consider an ideal \(\I \subseteq C_0(X) \rtimes \Z\). Since \(\varphi\) is aperiodic,  \(\I \cong C_0(U) \rtimes \Z\), where \(U\) is the open subset of \(X\) outside of which the elements of \(\I \cap C(X)\) vanish. The Pimsner-Voiculescu sequence associated to the crossed product \(\I = C_0(U) \rtimes \Z\): 
		\[\xymatrix{	& K_0(C_0(U)) \ar[r]^-{1-\varphi_*} & K_0(C_0(U)) \ar[r]  & K_0(\I) \ar[d] \\
			& K_1(\I) \ar[u] & K_1(C_0(U)) \ar[l] & K_1(C_0(U)) \ar[l] }\]
		Since \(U\) is totally disconnected, this reduces to
		\[\xymatrix{	& C_0(U, \Z) \ar[r]^-{1-\varphi_*} & C_0(U, \Z) \ar[r]  & K_0(\I) \ar[d] \\
			& K_1(\I) \ar[u] & 0 \ar[l] & 0 \ar[l] }\]
		and the group homomorphism \(1-\varphi_*\) is given by \[(1-\varphi_*)f(x) = f(x) - f(\varphi^{-1}(x)) \quad (f \in C_0(U, \Z), x \in U).\] If \( f \in \ker(1 - \varphi_*)\), the support of \(f\) is an invariant compact-open set, whence it follows that \(\ker(1 - \varphi_*)\) is non-trivial if and only if \(U\) contains an invariant compact-open set. Now, if \(K \subseteq U\) were an invariant compact-open set, topological transitivity would imply that \(K = X\). Hence, \(\ker(1 - \varphi_*)\) is trivial as soon as \(U\) is a proper subset of \(X\), and in the case where \(U = X\), the only invariant clopen set is \(X\) and so every element of \(\ker(1 - \varphi_*)\) is constant. 
	\end{proof} 
	
	\subsection{The largest AF-ideal}
	
	The following characterizes when a crossed product of a non-compact, totally disconnected spaces by a free action of the integers is AF. This follows by essentially combining Poon's theorem (\cref{thm_poon}) with \cref{prop_K_ideals}. 
	
		\begin{theorem}
		\label{thm_poon_deluxe}
		Let \(X\) be a locally compact, non-compact, totally disconnected space and let \(\varphi \colon X \to X\) be an aperiodic homeomorphism. Consider the following conditions: 
		\begin{enumerate}
			\item \(C_0(X) \rtimes  \Z\) is AF. 
			\item \(C_0(X) \rtimes  \Z\) is LF. 
			\item The \(K_1\)-group of every quotient of \(C_0(X) \rtimes  \Z\) is trivial. 
			\item For every compact-open set \(K \subseteq X\), we have \(\bigcap_{n \in \Z} \varphi^n(K) = \emptyset\). 
			\item Whenever \(U \subseteq X\) is a proper invariant open set, \(X \setminus U\) contains no non-empty invariant compact-open set. 
			\item Whenever \(U \subseteq X\) is a proper invariant open set, \(X \setminus U\) is non-compact and the action on \(X \setminus U\) is topologically transitive. 
			\item Whenever \(U \subseteq X\) is a proper invariant open set, \(X \setminus U\) is non-compact and contains a point whose orbit is dense. 
		\end{enumerate}
		Conditions (i)-(vi) are equivalent and (vii) implies (vi). If in addition, \(X \setminus U\) contains an isolated point for every proper invariant open set \(U \subseteq X\), we have \((vi) \implies (vii)\). 
	\end{theorem}
	\begin{proof}
		The implication (vii) \(\Rightarrow\) (vi) is well-known and (vi) \(\Rightarrow\) (v) is clear. That (vi) implies (vii) assuming the existence of isolated points is mentioned in \cite[Proposition 5.3]{akin}, but we include the short argument for the reader's convenience: Suppose that \(x \in X \setminus U\) is an isolated point. Since \(\varphi\) is a homeomorphism, \(\varphi^n(x)\) is isolated for each \(n \in \Z\) and so the orbit of \(x\) is an open invariant set, which is then dense by topological transitivity. 
		
		We now concentrate on the equivalence of the first six conditions. That (i) implies (ii) is clear. Secondly, since every quotient of an LF-algebra is LF, it follows that (ii) implies (iii). Moreover, (i) and (iv) are equivalent by Poon's theorem (\cref{thm_poon}) since (iv) is a reformulation of Poon's condition in the case \(Z = \{ \infty \}\) viewed as a subspace of the one-point compactification \(\tilde{X}\) of \(X\). Indeed, if \(W\) is a clopen neighborhood of \(\infty\) in \(\tilde{X}\), it contains a basic neighborhood of \(\infty\), i.e., one of the form \((X \setminus K) \cup \{ \infty \}\) for a compact set \(K \subseteq X\). Since \(X\) is totally disconnected, \(K\) can be covered by compact-open sets and so we may assume that \(K\) itself is compact-open at the cost of making \(X \setminus K\) even smaller. Using this, it is immediate that (iv) is equivalent to Poon's condition. 
		
		That (iii) implies (iv) is essentially the same argument as in the proof of \cref{thm_poon}. We include it here for completeness: If \(K \subseteq X\) is a compact open set such that \(\bigcap_{n \in \Z} \varphi^n(K) \neq \emptyset\), then \(W = (X \setminus K) \cup \{ \infty \}\) is a clopen set in \(\tilde{X}\) such that \(\bigcup_{n \in \Z} \varphi^n(W)\) is a proper invariant open subset. Set \(Y = X \setminus \bigcup_{n \in \Z} \varphi^n(W)\). Then the restriction map \(C(\tilde{X}) \to C(Y)\) induces a surjective \(*\)-homomorphism \[\pi \colon A_{ \{\infty \}} \to C(Y) \rtimes  \Z,\] where \(A_{ \{\infty \}}\) is the C*-subalgebra of \(C(\tilde{X}) \rtimes  \Z\) generated by \(C(\tilde{X})\) and \(uC_0(X)\), which is isomorphic to the unitization of \(C_0(X) \rtimes  \Z\). Now by the proof of \cref{prop_K_ideals}, the \(K_1\)-group of \(C(Y) \rtimes  \Z\) is non-trivial, since it is isomorphic to the \(\varphi\)-invariant integer-valued functions on \(Y\), which at least contains the constants. 
		
		It follows from the proof of \cref{prop_K_ideals} that (iii) and (v) are equivalent. Indeed, every quotient of \(C_0(X) \rtimes  \Z\) is of the form \(C_0(X \setminus U) \rtimes  \Z\) for some invariant open set \(U \subseteq X\). By the Pimsner-Voiculescu exact sequence, we see that \[K_1(C_0(X \setminus U) \rtimes  \Z) \cong \ker(1 - \varphi_*)\] in the notation of the proof of \cref{prop_K_ideals}. This group vanishes if and only if there are no invariant compact-open subsets of \(X \setminus U\), as in the proof of \cref{prop_K_ideals}.
		
		At last, (vi) implies (v): Consider any non-compact space \(Y\) with a topologically transitive action. If \(K\) is an invariant compact-open set, then \(Y \setminus K\) is open and so there is \(k\) such that \(\varphi^k(K)\) intersects \(Y \setminus K\) non-trivially by topological transitivity, but this contradicts that \(K\) is invariant. 
	\end{proof}
	
	\begin{remark}
		The implication (vi) \(\Rightarrow\) (vii) does not hold in general, see \cite[Example 7.4]{akin}.
	\end{remark}
	
	As we saw in \cref{subsec_AF}, every C*-algebra has a greatest LF-ideal. In this section, we use \cref{thm_poon_deluxe} to identify the largest LF-ideal in certain crossed products in dynamical terms and observe that it is in fact AF. Interestingly, it coincides with the intersection of all maximal ideals in the crossed product, whence it follows that the quotient by the largest AF-ideal embeds into the direct product of all simple quotients of the original algebra. 
	
	Consider the family \((Y_\alpha)_{\alpha \in \bbA}\) of minimal closed invariant sets for \(\varphi\) in \(X\). Let \(Y\) denote the closure of \(\bigcup_{\alpha \in \bbA} Y_\alpha\) in \(X\) and set \(U = X \setminus Y\). It follows that \(U\) contains no non-empty compact invariant set for \(\varphi\). Indeed, if \(K \subseteq U\) were such a set, it would be be a closed invariant subset of \(X\). Hence, it would contain \(Y_\alpha\) for some \(\alpha \in \bbA\), which contradicts that \(K \subseteq U\). Then it follows from item (v) in \cref{thm_poon_deluxe} that the ideal \(\I=C_0(U) \rtimes \Z\) in \(C(X) \rtimes \Z\) is AF. In fact, \(\I\) is the greatest LF-ideal: If \(V \subseteq X\) is any open invariant subset such that \(C_0(V) \rtimes \Z\) is LF, then \(V\) cannot contain any compact invariant sets. Indeed, if it did, then it would have a quotient with non-trivial \(K_1\)-group, contradicting item (iii) in \cref{thm_poon_deluxe}. It follows that \[\bigcup_{\alpha \in \bbA} Y_\alpha \subseteq X \setminus V \implies Y \subseteq X \setminus V,\] which is to say that \(V \subseteq U\). We collect this and some other easy observations about \(C_0(U) \rtimes \Z\) and the corresponding quotient in the following theorem. 
	
	\begin{theorem}
		\label{thm_maxAF}
		Let \(X\) be a compact Hausdorff space and consider an aperiodic homeomorphism \(\varphi \colon X \to X\). Let \(U\), \((Y_\alpha)_{\alpha \in \bbA}\) and \(Y\) be as described above. The ideal \(C_0(U) \rtimes \Z\) in \(C(X) \rtimes \Z\) is AF, and the greatest LF-ideal in \(C(X) \rtimes \Z\), as well as the intersection of all maximal ideals in \(C(X) \rtimes \Z\). Moreover, the following hold: 
		\begin{enumerate} 
			\item The quotient \(C(Y) \rtimes \Z\) embeds into \[\prod_{\alpha \in \bbA} C(Y_\alpha) \rtimes \Z,\] and this inclusion is proper as soon as \(\bigcup_\alpha Y_\alpha \neq Y\). 
			\item The space \(Y\) is perfect, i.e., contains no isolated points. 
		\end{enumerate}
	\end{theorem}
	
	\begin{proof}
		The maximal ideals in \(C(X) \rtimes \Z\) are \(C_0(X \setminus Y_\alpha) \rtimes \Z\) for \(\alpha \in \bbA\), and so it is clear that \[C_0(U) \rtimes \Z \subseteq \bigcap_{ \alpha \in \bbA} C_0(X \setminus Y_\alpha) \rtimes \Z.\] For the other inclusion, consider an element \(x\) in the intersection on the right-hand side. Its Fourier coefficients all have to be supported on \(X \setminus \bigcup_{\alpha \in \bbA} Y_\alpha\). Since the action is free, the reverse inclusion follows from \cite{sierakowski}. The map \(C(Y) \rtimes \Z \to \prod_{\alpha \in \bbA} C(Y_\alpha) \rtimes \Z\) is the product of the restriction maps \(C(Y) \rtimes \Z \to C(Y_\alpha) \rtimes \Z\) for \(\alpha \in \bbA\), and it is injective since \(\bigcup_{\alpha \in \bbA} Y_\alpha\) is dense in \(Y\). For the second part of (i), we argue contrapositively: Suppose that the inclusion is actually an isomorphism. Then the maximal ideals in \(\prod_{\alpha} C(Y_\alpha) \rtimes \Z\) are \(\prod_{\beta \neq \alpha} C(Y_\beta) \rtimes \Z\) for \(\alpha \in \bbA\), all of which are unital. The corresponding maximal ideals in \(C(Y) \rtimes \Z\) are \(C_0(Y \setminus Y_\alpha) \rtimes \Z\) for \(\alpha \in \bbA\). It follows then that \(Y \setminus Y_\alpha\) is compact and that \(Y_\alpha\) is open in \(Y\) for every \(\alpha \in \bbA\). Hence, \(\bigcup_{\alpha \in \bbA} Y_\alpha\) is an invariant open set and so \(Y \setminus \bigcup_\alpha Y_\alpha\) is a compact invariant set that admits no minimal closed invariant set. Hence, it is empty and we conclude that \(Y = \bigcup_\alpha Y_\alpha\). 
		
		Towards (ii), assume for contradiction that \(x \in Y\) is an isolated point. Since \(\bigcup_\alpha Y_\alpha\) is dense in \(Y\) and \(\{x\}\) is open, it follows that \(x \in Y_\alpha\) for some \(\alpha\). Since \(x\) is isolated in \(Y\), there is an open set \(U \subseteq X\) such that \(U \cap Y = U \cap Y_\alpha = \{x\}\), and so \(x\) is also isolated in \(Y_\alpha\), but this contradicts minimality of \(Y_\alpha\). 
	\end{proof}
	
	\begin{remark} 
		\begin{enumerate}
			\item If \(\bbA\) is a finite set, then the embedding is of course an isomorphism, and we conclude that $C(X) \rtimes \Z$ is an extension of an AF-algebra by a (locally) A$\T$-algebra of real rank zero with vanishing exponential map, by \cref{prop_K_ideals}. It follows that $C(X) \rtimes \Z$ has real rank zero. This is known, by a similar inductive extension argument, in \cite[Corollary 7.11]{bezuglyi}. More generally, it was recently shown by An and Liu that any free action of \(\Z\) on a totally disconnected space yields a crossed product of real rank zero, see \cite[Corollary 4.3]{an}.
			\item It was shown in \cite[Theorem 8.9]{herman}, using Zorn's lemma, that there is a closed invariant subset \(Y\) such that the corresponding ideal is the largest AF-ideal in \(C(X) \rtimes \Z\), without providing this explicit description. In the presence of the ideal separation property, the existence of such a set follows from \cref{prop_LF}.
		\end{enumerate}
	\end{remark}
	
	In the case of the uniform Roe algebra of \(\Z\), i.e., when \(X = \beta \Z\), the space \(Y\) described above is well-known in topological semigroup theory. In this case, each \(Y_\alpha\) is a minimal left ideal in the semigroup \(\beta \Z\). They are all isomorphic as dynamical systems to the universal minimal \(\Z\)-system, see for example \cite[Theorem 19.8]{hindman_strauss_alg}, and each \(Y_\alpha\) is a perfect Stonean space. Their union \(\bigcup_{\alpha \in \bbA} Y_\alpha\) is the smallest two-sided ideal in the semigroup \(\beta \Z\), often denoted \(K(\beta \Z)\). Thinking of the points of \(\beta \Z\) as ultrafilters on \(\Z\), it is well-known that an ultrafilter \(p\) lies in \(K(\beta \Z)\) if and only if every \(A \in p\), the set \(\{x \in \Z : A-x \in p\}\) is syndetic, and \(p\) lies in \(Y = \overline{K(\beta \Z)}\) if and only if  every \(A \in p\) is piecewise syndetic, see \cite[Theorem 4.40]{hindman_strauss_alg}.
	
	Isolated points give rise to minimal projections in the crossed product. As seen in the next proposition, any projection that admits a finite-dimensional representation corresponds to a set of isolated points in some closed invariant subset. 
	
	\begin{lemma}
		\label{lemma_proj}
		Let \(\Gamma\) be an exact group acting freely on a totally disconnected, compact space \(X\). Every ideal in \(C(X) \rtimes_r \Gamma\) that is generated by a projection is generated by a projection in \(C(X)\). 
	\end{lemma}
	
	\begin{proof}
		This is similar to the proof \cite[Proposition 5.3]{kellerhals}. Consider an ideal \(\I\) in \(C(X) \rtimes_r \Gamma\) that is generated by a projection \(q\). By \cite[Theorem 1.6]{sierakowski}, \(\I\) is generated by \(\I \cap C(X)\). Consider the directed set \(\Lambda\) of finite subsets of projections in \(\I \cap C(X)\). For each \(\lambda \in \Lambda\), let \(\I_\lambda\) denote the ideal generated by the projections in \(\lambda\). Then clearly \(\I\) is the closure of the union \(\bigcup_{\lambda \in \Lambda} \I_\lambda \). Hence, \(\bigcup_{\lambda \in \Lambda} \I_\lambda \) contains the Pedersen ideal of \(\I\), see \cite[Section 5.6]{pedersen_blue}. Since the Pedersen ideal contains all projections of \(\I\), there must be some \(\lambda \in \Lambda\) such that \(I_\lambda\) contains the generator \(q\). Hence, \(\I = \I_\lambda\) and \(\I\) is generated by the supremum of the projections in \(\lambda\). 
	\end{proof}
	
	\begin{proposition}
		\label{prop_proj}
		Let \(\Gamma\) be an amenable group acting freely on a totally disconnected, compact space \(X\), and consider the associated crossed product \(C(X) \rtimes \Gamma\). The following are equivalent. 
	\begin{enumerate}
			\item \(C(X) \rtimes \Gamma\) admits a representation mapping a projection to a non-zero finite-rank projection. 
			\item \(C(X) \rtimes \Gamma\) admits an irreducible representation mapping a projection to a non-zero finite-rank projection. 
			\item There is a non-empty clopen set \(A \subseteq X\) and a closed invariant set \(K \subseteq X\) such that \(A \cap K\) is a finite set of isolated points in \(K\).
			\item \(C(X) \rtimes \Gamma\) admits a representation mapping a diagonal projection to a non-zero finite-rank projection. 
		\end{enumerate}
	\end{proposition}
	
	\begin{proof}
		It is clear that (iv) implies (i). That (i) implies (ii) holds for general C*-algebras, but we present the argument for completeness: Consider a representation \(\pi \colon \A \to \B(H)\) of a C*-algebra, and suppose that there is a projection \(p \in \A\) such that \(\pi(p)\) is a non-zero finite-rank projection. Then we have a finite-dimensional representation \(\pi \colon p \A p \to \B(\pi(p)H)\) of the corner \(p \A p\) by compression. Since \(\pi(p)H\) is finite-dimensional, this representation decomposes into a finite sum of irreducible representations. One summand of this, call it \(\pi_0\), must map \(p\) to a non-zero finite-rank projection. It is well-known that irreducible representations extend to irreducible representations, see \cite[Theorem 5.5.1]{murphy}, and so we may take any irreducible extension of \(\pi_0\) to \(\A\). 
		
		We now show that (ii) implies (iii). Consider a representation \(\pi \colon C(X) \rtimes \Gamma \to \B(H)\). We first show that if there exists a projection \(p \in C(X) \rtimes \Gamma\) such that \(\pi(p)\) has non-zero, finite rank, there exists a diagonal projection with the same property. Indeed, let \(\I\) denote the ideal generated by \(p\). By \cref{lemma_proj}, \(\I\) is also generated by a diagonal projection \(q\). Since \(\pi(\I) \subseteq \K(H)\), we conclude that \(\pi(q)\) must be a finite-rank projection, and it must be non-zero since \(\pi(\I)\) is. Let \(p \in C(X)\) be a projection, i.e., \(p = 1_A\) for some clopen set \(A \subseteq X\). If \(\pi(p)\) is of finite rank, then we get a finite-dimensional representation of \(p (C(X) \rtimes \Gamma) p\) by compression. In other words, \(p (C(X) \rtimes \Gamma) p\) has a finite-dimensional quotient. It follows that there is a quotient of $C(X) \rtimes \Gamma$ with a finite-dimensional corner, i.e., there is a closed invariant set \(K \subseteq X\) and a projection \(q \in C(K) \rtimes \Gamma\) such that \(q (C(K) \rtimes \Gamma) q\) is finite-dimensional. In fact, one may choose \(q = 1_{K \cap A}\) and conclude that \(C(K \cap A) \subseteq q (C(K) \rtimes \Gamma) q\) is finite-dimensional. This means that \(A \cap K\) is a finite set consisting of isolated points. 
		
		Next, we show that (iii) implies (iv). Clearly, it suffices to show that if \(X\) admits an isolated point, then \(C(X) \rtimes \Gamma\) admits a representation mapping \(1_{\{x\}}\) to a one-dimensional representation. Indeed, let \(\pi_x \colon C(X) \to \B(\ell^2(\Gamma))\) be given by \( \left( (\pi_x)(f)\xi \right) (g) = f(gx) \xi (g)\) for \(f \in C(X)\) and \(g \in \Gamma\). Then \((\pi, \lambda)\) is a covariant representation on \(\ell^2(\Gamma)\) of the action of \(\Gamma\) on \(X\), where \(\lambda\) denotes the left-regular representation of \(\Gamma\). The integrated form \(\pi_x \rtimes \lambda \colon C(X) \rtimes \Gamma \to \B(\ell^2(\Gamma))\) maps \(1_{\{x\}}\) to the orthogonal projection onto the subspace of vectors supported on the stabilizer subgroup of \(x\), which is one-dimensional since the action is free.
	\end{proof}
	
	\begin{corollary}
		\label{cor_iso}
		Let \(\Gamma\) be an amenable group acting freely on a totally disconnected, compact space \(X\). Then the crossed product \(C(X) \rtimes \Gamma\) admits a separating family of irreducible representations meeting the compacts if and only if \(X\) contains a dense set of isolated points. 
	\end{corollary}
	
	\begin{proof}
		Let \(\I\) denote the intersection of all kernels of irreducible representations of \(C(X) \rtimes \Gamma\) meeting the compacts, and let \(U \subseteq X\) denote the (possibly empty) open invariant set such that \(\I = C_0(U) \rtimes \Gamma\). By \cref{prop_proj}, \(U \neq \emptyset\) if and only if it contains a non-empty compact-open set with no isolated points. Hence, \(\I = 0\) if and only if the set of isolated points in \(X\) is dense. 
	\end{proof}
	
	Looking at the proof of (ii) implies (iii) also yields the following porism. 
	
	\begin{corollary}
		Let \(\Gamma\) be an amenable group acting freely on a totally disconnected, compact space \(X\), and consider a diagonal projection \(1_A \in C(X) \subseteq C(X) \rtimes \Gamma\). Then \(C(X) \rtimes \Gamma\) admits a representation \(\pi\) mapping \(1_A\) to a non-zero finite-rank projection if and only if there is a closed invariant set \(K\) such that \(A \cap K\) consists of \(\tr(\pi(1_A))\) points isolated in \(K\). 
	\end{corollary}
	
	A C*-algebra all of whose irreducible representations meet, and hence contain, the compacts is called \textit{type I} (alternatively, \textit{postliminal} or \textit{GCR}). Equivalently, a C*-algebra is type I if and only if each of its non-zero quotients contains a non-zero abelian element, i.e., an element generating a non-zero abelian hereditary C*-subalgebra. It is well-known that any C*-algebra admits a largest type I ideal, see \cite[Proposition 6.2.7]{pedersen_blue}. It is also well-understood when (universal) crossed products are type I, namely precisely if the orbit space is \(T_0\) and all stabilizer subgroups are type I groups, see \cite[Theorem 3.3]{gootman}. Having seen \cref{thm_maxAF} and \cref{prop_proj}, it seems interesting to compare the largest LF-ideal and largest ideal of type I. In fact, one can show that the largest type I ideal in $C(X) \rtimes \Gamma$ is LF, and hence contained in the largest LF-ideal whenever \(\Gamma\) is an exact group acting freely on a totally disconnected space \(X\). 
	
	\begin{proposition}
		\label{prop_inclusion}
		Let \(\Gamma\) be an exact group acting freely on a totally disconnected space \(X\). Then the largest type I-ideal in \(C(X) \rtimes_r \Gamma\) is LF. 
	\end{proposition}
	
	\begin{proof}
		For any C*-algebra \(\A\), let \(\sC(\A)\) denote the largest type I-ideal in \(\A\). Then, as in \cref{lemma_approx}, we may approximate $C(X) \rtimes \Gamma$ as an increasing union of separable crossed products \[C(X) \rtimes \Gamma = \overline{ \bigcup_i C(X_i) \rtimes \Gamma}\] and it follows that \[\sC(C(X) \rtimes \Gamma) = \overline{ \bigcup_i C(X_i) \rtimes \Gamma \cap \sC(C(X) \rtimes \Gamma)}.\] Being type I passes to subalgebras, whence each \(C(X_i) \rtimes \Gamma \cap \sC(C(X) \rtimes \Gamma)\) is type I. Moreover, it is an ideal in \(C(X_i) \rtimes \Gamma\). It is easy to check that any crossed product \(\A \rtimes \Gamma\), and approximate unit for \(\A\) is an approximate unit for \(\A \rtimes \Gamma\). Hence, each ideal in \(C(X_i) \rtimes \Gamma \cap \sC(C(X) \rtimes \Gamma)\) has an approximate unit consisting of projections. Now by \cite[Proposition 7.14]{pasnicu_phillips}, we conclude that \(C(X_i) \rtimes \Gamma \cap \sC(C(X) \rtimes \Gamma)\) is AF. Hence, \(\sC(C(X) \rtimes \Gamma)\) is LF. 
	\end{proof}
	
	\begin{remark}
		Note that Lazar and Taylor gave an example of an antiliminal AF-algebra with a separating family of finite-dimensional irreducible representations in \cite[Example 4.4]{lazar}. Hence, \cref{prop_proj} is of little use in understanding when the opposite inclusion holds. 
	\end{remark}
	
	Returning to the case of the uniform Roe algebra of \(\Z\), in light of the results of this section, it is pertinent to ask if there are closed invariant subsets of \(Y\) that admit isolated points. In general, it is a hard problem to understand the topology of closed invariant subsets of \(Y\). However, those defined as orbit closures of elements that are right cancelable in the semigroup \(\beta \Z\) are known to be homeomorphic to \(\beta \Z\) (in fact this is true for general semigroups), see \cite[Theorem 4.7]{hindman}. It is known that for any group, every non-minimal orbit closure in the Stone--\v{C}ech compactification contains an abundance of these, see \cite[Corollary 8.26]{hindman_strauss_alg}. Hence, we conclude that \(Y\) contains closed invariant subsets that are homeomorphic to \(\beta \Z\). They are orbit closures, where the orbit consists of isolated points. Hence, the quotient \(C(Y) \rtimes \Z\) still has an abundance of representations that meet the compacts. Moreover, clearly the standard representation of \(C_u^*(\Z)\) on \(\ell^2(\Z)\) is a faithful irreducible representation meeting the compacts. However, since the Stone--\v{C}ech boundary \(\beta \Z \setminus \Z\) is perfect, \cref{cor_iso} implies that \(C_u^*(\Z)/ \K\) does not admit a separating family of irreducible representations meeting the compacts. However, as we just argued, it has further quotients with this property. 
	
	By \cite[Theorem 3.1]{putnam}, \(C(Y) \rtimes \Z\) does not have cancellation since \(Y\) contains (way) more than one minimal closed invariant set. However, \cref{thm_maxAF} above implies that it is finite since it embeds in a finite C*-algebra. Even more, the product \(\prod_{\alpha \in \bbA} C(Y_\alpha) \rtimes \Z\) has a separating family of traces, since each factor is simple with a (necessarily) faithful trace. However, by the comments above and \cref{prop_proj}, there is an irreducible representation \(\pi \colon C(Y) \rtimes \Z \to \B(H)\) such that \(\K(H) \subseteq \operatorname{im} (\pi)\). Then \(C(Y) \rtimes \Z / \ker(\pi)\) is a quotient of \(C(Y) \rtimes \Z\) that does not admit a separating family of traces (since every trace would vanish on the copy of the compacts). Consequently, this quotient cannot embed into any product of the \(C(Y_\alpha) \rtimes \Z\).
	
	\section{Uniform Roe algebras}
	
	In this section, we give a coarse geometric characterization of when the compact ideals in uniform Roe algebras of (exact) groups are AF. Following \cite{kellerhals}, given a subset \(A \subseteq \Gamma\), we consider the open invariant subspace \[X_A = \bigcup_{g \in \Gamma} gK_A = \bigcup_{g \in \Gamma} K_{gA}\] of \(\beta \Gamma\). The associated ideal \(C_0(X_A) \rtimes_r \Gamma\) of \(C_u^*(\Gamma)\) is denoted by \(\I_A\), and it is a \textit{compact} ideal in the sense that if \((\I_\lambda)_{\lambda \in \Lambda}\) is an increasing net of ideals whose union is dense in \(\I_A\), then \(\I_A = \I_\lambda\) for some \(\lambda \in \Lambda\). We will show in the following proposition that these are precisely the compact ideals in \(C_u^*(\Gamma)\), and every ideal in \(C_u^*(\Gamma)\) is an increasing union of compact ideals. In that sense, the understanding of many regularity properties for ideals in \(C_u^*(\Gamma)\) can be reduced to understanding that regularity property for the compact ideals in terms of the generating subsets. Following the notation and terminology from \cite{kellerhals}, for two subsets \(A, B \subseteq \Gamma\), we say that \(A\) is \(B\)\textit{-bounded}, denoted \(A \gbd B\), if there is a finite subset \(F \subseteq \Gamma\) such that \[A \subseteq \bigcup_{g \in F} gB.\] It is not hard to show that \(K_A \subseteq X_B\) if and only if \(A \gbd B\) and that \(X_A = X_B\) if and only if \(A \gbd B \gbd A\), see \cite[Lemma 2.5]{kellerhals}. 
	
	Compact ideals in a C*-algebra \(\A\) correspond to compact subsets of the primitive ideal space of \(\A\). In general, ideals generated by finitely many projections are compact, and in nice cases, like below and in the purely infinite setting, see \cite[Proposition 2.7]{pasnicu_rordam}, one can show that any compact ideal is generated by a single projection. 
	
	\begin{proposition}
		\label{prop_compact}
		Let \(\Gamma\) denote an exact group. The compact ideals in \(C_u^*(\Gamma)\) are precisely the ideals of the form \(\I_A\) for subsets \(A \subseteq \Gamma\). Moreover, every ideal in \(C_u^*(\Gamma)\) is an increasing union of such compact ideals. 
		
		For two subsets \(A, B \subseteq \Gamma\), we have the following: 
		\begin{enumerate}
			\item \(\I_A \subseteq \I_B\) if and only if \(A \gbd B\). 
			\item \(\I_{A \cup B} = \I_A + \I_B\). 
		\end{enumerate}
	\end{proposition}
	
	\begin{proof}
		Let us first take care of the latter two statements. Clearly, \(\I_A \subseteq \I_B\) if and only if \(K_A \subseteq X_B\), and so item (i) follows from \cite[Lemma 2.5]{kellerhals}. For the second, statement, note that since \(X_{A \cup B}\) contains \(X_A\) and \(X_B\), we have that \(\I_A + \I_B \subseteq \I_{A \cup B}\). Conversely, \(1_{A \cup B} = 1_A + 1_B - 1_{A \cap B}\). Since \(1_A \in \I_A\) and \(1_B - 1_{A \cap B} \in \I_B\), we conclude that \(1_{A \cup B} \in \I_A + \I_B\) and hence \(\I_{A \cup B} \subseteq \I_A + \I_B\). 
		
		For the first statement, clearly \(\I_A\) is compact since it is generated by a single projection, namely \(1_A\). Conversely, suppose that \(\I\) is a compact ideal, and let \(U\) denote the open invariant subspace of \(\beta \Gamma\) such that \(\I = C_0(U) \rtimes_r \Gamma\). Then \(U\) is still totally disconnected, and the set of compact-open subsets of \(U\) is directed by inclusion. For any compact-open subset \(K \subseteq U\), let \(\I_K\) denote the ideal generated by \(1_K\). Clearly, the union of all \(\I_K\) as \(K\) runs through the compact-open subsets of \(U\) is dense in \(\I\). Thus compactness implies that there is a compact-open set \(K\) such that \(\I = \I_K\). Since \(U\) is open in \(\beta \Gamma\), \(K\) is clopen in \(\beta \Gamma\) and thus there is a subset \(A \subseteq \Gamma\) such that \(K = K_A\). Running the same argument without the assumption that \(\I\) is compact shows that any ideal in \(C_u^*(\Gamma)\) is the closed union of an increasing net of compact ideals. 
	\end{proof}
	
	The proof of the following theorem leans heavily on the proof of \cite[Theorem 2.2]{li2018low}. 
	
	\begin{theorem}
		\label{thm_AF_ideals}
		Let \(A \subseteq \Gamma\) be a subset. The following are equivalent: 
		\begin{enumerate}
			\item The set \(A\) has asymptotic dimension zero. 
			\item The ideal \(\I_A\) is AF. 
			\item The ideal \(\I_A\) is LF. 
			\item The ideal \(\I_A\) has stable rank one. 
			\item The ideal \(\I_A\) has cancellation. 
		\end{enumerate} 
	\end{theorem}
	
	\begin{proof}
	First note that the implications 
	\[(ii) \implies (iii)\ \implies (iv) \implies (v)\] all hold in general. To see that (v) implies (i), observe that if \(A\) does not have asymptotic dimension zero, then \(1_A C_u^*(\Gamma) 1_A \cong C_u^*(A)\) does not have cancellation by \cite[Theorem 2.2]{li2018low}, and hence neither does \(\I_A\) since cancellation passes to corners. 
	
	In order to show that (i) implies (ii), we follow and adapt the argument of (i) implies (ii) in \cite[Theorem 2.2]{li2018low}, largely following the same notation. For the convenience of the reader, we provide the full construction. Fix a total order on \(\Gamma\), and for each finite subset \(S \subseteq \Gamma\), let \(f_S \colon S \to \{1, \ldots, \abs{S}\}\) be the corresponding order-isomorphism. Let \(\Lambda\) denote the set of triples \((F, r, \P)\), where \(F\) is a finite subset of \(\Gamma\), \(r > 0\) and \(\P = \{P_1, \ldots, P_N\}\) is a finite partition of the set \(I^{FA}_r\) of equivalence classes of the relation \(\sim_r\) on the space \(FA := \bigcup_{g \in F} gA\) from \cref{prop_equiv_rel} such that any two equivalence classes in the same element of \(\P\) have the same cardinality. Now, consider such a triple \((F, r, \P)\), and let \(n_1, \ldots, n_N\) be the cardinalities of the sets \(P_1, \ldots, P_N\) in \(\P\) respectively. For each \(S \in P_i\), let \(u_{S, i} \colon \C^{S} \to \ell^2(S)\) denote the unitary map determined by \(f_S\). Define an embedding 
	\begin{align*} 
		\varphi_{(F, r, \P)} \colon \bigoplus_{i=1}^N M_{n_i}(\C) & \to \prod_{S \in I_r^{FA}} \B(\ell^2(S)) \subseteq \B(\ell^2(\Gamma)), \\
		(a_i)_{i=1}^N & \mapsto \prod_{i=1}^N \prod_{S \in I_r^{FA}} u_{S, i} a_i u_{S, i}^*. 
	\end{align*}
	Note that since \(\sim_r\) has uniformly bounded equivalence classes, see \cref{prop_equiv_rel}, the operator \(\varphi_{(F, r, \P)}(a)\) has finite propagation for any \(a \in \bigoplus_{i=1}^N M_{n_i}(\C)\) and hence the image \(\A_{(F, r, \P)}\) of \(\varphi_{(F, r, \P)}\) lies in \(C_u^*(\Gamma)\). Moreover, the operator is clearly supported on \(FA\), so it in fact lies in \(1_{FA} C_u^*(\Gamma) 1_{FA}\). The algebra \(\A_{(F, r, \P)}\) consists of operators supported on \(FA\) and that act as a single \((n_i \times n_i)\)-matrix on all the subspaces corresponding to sets in \(P_i\) for every \(i\). 
	
	We define a preorder on \(\Lambda\) by declaring that \((E, r, \P) \leq (F, s, \Q)\) if and only if \(\A_{(E, r, \P)} \subseteq \A_{(F, s, \Q)}\). We need to check that \(\Lambda\) is directed. For fixed \(E\) and \(r\), we have \((E, r, \P) \leq (E, r, \Q)\) if and only if \(\Q\) refines \(\P\). We claim that it suffices to show that for any index \((E, r, \P)\), any finite subset \(F\) containing \(E\) and \(s \geq r\), there is a partition \(\Q\) of the equivalence classes \(I^{FA}_s\) on \(FA\) such that \(\A_{(E, r, \P)} \subseteq \A_{(F, s, \Q)}\). Indeed, suppose that were established, and consider a pair of indices \((E_1, r_1, \P_1), (E_2, r_2, \P_2) \in \Lambda\). Put \(s = \max\{r_1, r_2\}\) and \(F = E_1 \cup E_2\), and use the claim to find partitions \(\Q_1\) and \(\Q_2\) of the equivalence classes of \(\sim_s\) on \(FA\) such that \(\A_{(E_1, r_1, \P_1)} \subseteq \A_{(F, s, \Q_1)}\) and \(\A_{(E_2, r_2, \P_2)} \subseteq \A_{(F, s, \Q_2)}\). Then, if \(\Q\) is a common refinement of \(\Q_1\) and \(\Q_2\), then \((F, s, \Q)\) is a common upper bound for \((E_1, r_1, \P_1)\) and \((E_2, r_2, \P_2)\). 
	
	To show the claim, consider one index \((E, r, \P)\), a finite subset \(F\) containing \(E\) and \(s \geq r\). Denote the elements of \(\P = \{P_1, \ldots, P_N\}\). For \(S \in I_s^{FA}\), let \(f_S \colon S \to \{1, \ldots, \abs{S}\}\) be as above, and define a set of subsets \(\P_S = \{P_{1, S}, \ldots, P_{N, S}\}\) of the set \(\{1, \ldots, \abs{S}\}\) by the property that \(k \in P_{i, S}\) if and only if \(k \in f_S(T)\) for some subset \(T \subseteq S\) such that \(T \in P_i\). We define the partition \(\Q\) of \(I_s^{FA}\) by saying that two sets \(S, T \in I_s^{FA}\) are in the same \(\Q\)-cell if and only if they have the same cardinality and \(P_{i, S} = P_{i, T}\) for \(i = 1, \ldots, N\). Now, somewhat informally, any operator in \(\A_{(E, r, \P)}\) is block-diagonal with respect to \(I_r^{EA}\) and hence also with respect to \(I_s^{FA}\), and since it is constant on \(\P\)-colors, it is on \(\Q\)-colors as well because of the way \(\Q\) is defined. Hence, \((E, r, \P) \leq (F, s, \Q)\). 
	
	It only remains to show that the union of all \(\A_{(E, r, \P)}\) over \((E, r, \P) \in \Lambda\) is dense in \(\I_A\). Since the argument in \cite{li2018low} already shows that \(\bigcup_{r, \P} \A_{(E, r, \P)}\) is dense in \(1_{EA} C_u^*(\Gamma) 1_{EA}\) for every finite subset \(E \subseteq \Gamma\), it suffices to show that the union of \(1_{EA} C_u^*(\Gamma) 1_{EA}\) over all finite subsets \(E \subseteq \Gamma\) is dense in \(\I_A\). To that end, it suffices to show that for every \(f \in C_c(X_A)\) and \(g \in \Gamma\), \[fu_g \in  \overline{ \bigcup_{F \Subset  \Gamma} 1_{FA} C_u^*(\Gamma) 1_{FA} },\] where \(F\) runs through all finite subsets of \(\Gamma\). To see that such is the case, consider the support of \(f\), which is a compact subset of \(X_A = \bigcup_{h \in \Gamma} K_{hA}\). Hence, there is some finite subset \(E \subseteq \Gamma\) such that \(\bigcup_{h \in E} K_{hA}\) contains the support of $f$. Let \(S = \bigcup_{h \in E} (hA \cup g^{-1}hA)\). Then \[1_S f u_g 1_S = 1_S f 1_{gS} u_g = f u_g \in 1_S C^*_u(\Gamma) 1_S \subseteq \overline{ \bigcup_{F \Subset \Gamma} 1_{FA} C_u^*(\Gamma) 1_{FA}},\] which is what we needed to show. 
	\end{proof}
	
	Notice that the last part of the argument above combined with \cite[Theorem 2.2]{li2018low} shows that \(\I_A\) is a direct limit of AF-algebras, which implies that it is LF. However, to see that it is in fact AF, we need the modification of their argument presented above. Now, it follows from the theorem above that the largest LF-ideal in \(C_u^*(\Gamma)\) is the closure of the union of \(\I_A\), where \(A\) runs through all asymptotic dimension zero subsets of \(\Gamma\). Of course, we already know by \cref{thm_maxAF} that in the case \(\Gamma = \Z\), this ideal is in fact AF. This will follow for the general case by appending a fourth factor to the indexing set in the proof above, namely the asymptotic dimension zero set \(A\). 
	
	Following up on the discussion in the previous section regarding the largest type I-ideal, it appears to be an interesting problem to characterize for which subsets \(A\) the ideal \(\I_A\) is type I. The related problem of characterizing when the orbit space of the uniform Roe groupoid is \(T_0\) is left open in \cite[Problem 2.26]{braga2026}. Note that in \cite[Theorem 2.25]{braga2026}, it is shown that if \(X\) has asymptotic dimension at least one, then the orbit space of the uniform Roe groupoid of \(X\) is not \(T_0\), which corresponds to the fact shown here that the largest type I ideal is AF. We offer some more insight towards these questions by identifying a property that guarantees that \(\I_A\) is type I and by providing an example of a subset \(A\) of the integers such that \(\I_A\) is AF, but not type I. Notice that if \(A\) is any non-empty finite subset of \(\Gamma\), then \(\I_A\) is \(\K(\ell^2(\Gamma))\) in the standard representation on \(\ell^2(\Gamma)\), which is the smallest ideal in \(C_u^*(\Gamma)\). We simply denote it by \(\K\) in the sequel. 
	
	\begin{definition}[{\cite{geller2026sparse}}]
		Let \((X, d)\) be a metric space. Fix a point \(x_0 \in X\) and consider the function \(s \colon (0, \infty) \to \R\) given by 
		\[s(r) = \inf \{d(x, y) : x, y \in X \setminus B_X(x_0, r), x \neq y\}.\]
		The metric space \(X\) is called \textit{sparse} if \(\lim_{r \to \infty} s(r) = \infty.\)
	\end{definition}
	
	\begin{remark}
		Clearly, the notion of sparseness is independent of the choice of \(x_0\) in the definition above. 
	\end{remark}
	
	We show in the following proposition that for subsets of discrete groups this notion coincides with a property studied used in \cite{kellerhals} to find abelian projections. In the subsequent proposition, we push this idea further to obtain a family of type I ideals in the uniform Roe algebra of the group. 
	
	\begin{lemma}
		Let \(\Gamma\) be a countable group and consider a subset \(A \subseteq \Gamma\). Then \(A\) is sparse as a metric space if and only if \(A \cap g A\) is finite for every \(g \neq e\). 
	\end{lemma}
	
	\begin{proof}
		First suppose that there is some \(g \neq e\) such that \(A \cap gA\) is infinite and put \(c = d(g, e) > 0\). Then for every \(r > 0\), since \(B_A(e, r)\) is finite and \(A \cap gA\) is infinite, \(A \setminus B_A(e, r)\) contains infinitely many pairs of points at distance \(c\). Hence, \(s(r) \leq c\) for every \(r\) and so \(A\) is not sparse. Conversely, suppose that \(A\) satisfies that \(A \cap gA\) is finite for every \(g \neq e\) and consider \(m > 0\). Any pair of points in \(A\) that are closer than \(m\) is of the form \(x, gx\) for some \(g \in \Gamma\) with \(d(g, e) \leq m\). Since \(\Gamma\) (and hence \(A\)) is uniformly locally finite, there are finitely many such \(g\). Moreover, by assumption, for every such \(g\), there are only finitely many pairs of elements in \(A\) of elements where one is obtained from the other by left multiplication by \(g\). Thus, there is an \(r > 0\) such that all distinct elements of \(A \setminus B_A(e, R)\) are at least \(m\) apart. Hence, \(s(r) \geq m\). Since \(m\) was arbitrary, we conclude that \(A\) is sparse. 
	\end{proof}
	
	Recall that a C*-algebra is \textit{type \(I_0\)} if it is generated as a C*-algebra by abelian elements. 
	
	\begin{proposition}
		If \(A \subseteq \Gamma\) is a finite union of sparse subsets of a countable group, then \(\I_A/ \K\) is type \(I_0\). In particular, \(\I_A\) is type I. 
	\end{proposition}
	
	\begin{proof} 
		Clearly the last statement follows from the first as an extension of two type I C*-algebras is again type I. Let \(\pi \colon \I_A \to \I_A/\K\) denote the quotient map. It suffices to assume that \(A\) is sparse and show that for any pair of elements \(g,h \in \Gamma\), finite subset \(F \subseteq \Gamma\) and corresponding \(f_k \in C_0(X_A)\) for \(k \in F\), we have 
		\begin{equation}
			\label{eq_sparse}
		\pi \left( (1_{hA}u_g)^* \left( \sum_{k \in F} f_k u_k \right) (1_{hA}u_g) \right) \in \pi(C_0(X_A))
		\end{equation}
		Indeed, that would show that the element \(1_{hA}u_g\) is abelian modulo \(\K\). Moreover, if \(C\) is any compact-open set contained in \(hA\), the hereditary C*-subalgebra generated by \(1_Cu_g\) is contained in the one generated by \(1_A u_g\), and so we conclude that for all such \(C\), the element \(1_Cu_g\) is abelian modulo \(\K\). Letting these run through all compact-open sets of \(X_A\) and all \(g \in \Gamma\), we conclude that \(\I_A\) is generated as a C*-algebra by elements that are abelian modulo \(\K\), assuming that \(A\) is sparse. If \(A\) is a finite union of sparse sets, we combine the previous case with \cref{prop_compact}.
		
		To verify \cref{eq_sparse}, we expand the brackets in the element inside \(\pi\): 
		\begin{align*}
			(1_{hA}u_g)^* \left( \sum_{k \in F} f_k u_k \right) (1_{hA}u_g) &= \sum_{k \in F} u_g^* 1_{hA} f_k u_k 1_{hA} u_g \\
			&= \sum_{k \in F} u_g^* 1_{hA} f_k 1_{khA} u_{kg} \\
			&= \sum_{k \in F} 1_{g^{-1} hA} \alpha_g^{-1}(f_k) 1_{g^{-1}khA} u_{g^{-1}k g} \\
			&= \sum_{k \in F} 1_{g^{-1}h(A \cap h^{-1}khA)} \alpha_g^{-1}(f_k) u_{g^{-1}kg}
		\end{align*}
		By assumption, \(A \cap h^{-1}khA\) is finite unless \(k = e\). Hence, 
		\[	\pi \left( (1_{hA}u_g)^* \left( \sum_{k \in F} f_k u_k \right) (1_{hA}u_g) \right) = \pi(1_{g^{-1}hA} \alpha_g^{-1}(f_e) u_e) \in \pi(C_0(X_A)),\]
		which completes the proof. 
	\end{proof}
	
	 Although it is well-known that direct limits of type I C*-algebras need not be type I, the closure of an increasing directed union of ideals of type I is still type I. We could not find a reference to this fact in the literature, so we record it with a proof below for completeness. Clearly, a union of two sets that are finite unions of sparse sets is still a finite union of sparse sets, and thus the family of subsets of \(\Gamma\) that are finite unions of sparse sets, is directed by inclusion. Hence, the closed union of compact ideals associated to such subsets of \(\Gamma\) is an ideal in \(C_u^*(\Gamma)\) of type I. It remains open if this ideal is the largest type I ideal in general. 
	 
	 \begin{proposition}
	 	Let \(\A\) be a C*-algebra and let \((\I_\lambda)_{\lambda \in \Lambda}\) be a directed family of type I ideals in \(\A\). Then the ideal \[\I := \overline{\bigcup_{\lambda \in \Lambda} \I_\lambda}\] is of type I. 
	 \end{proposition}
	 
	 \begin{proof}
	 	Consider a proper ideal \(\J \subseteq \I\). Since \(\J\) is proper, there must be some \(\lambda\) such that \(\I_{\lambda}\) is not contained in \(\J\). Hence, \(\I_{\lambda}/J \cap \I_{\lambda}\) contains a non-zero abelian element \(x\). Moreover,  \(\I_{\lambda}/\J \cap \I_{\lambda}\) sits as an ideal in \(\I/\J\) and so \(x\) is an abelian element in \(\I/\J\). 
	 \end{proof}
	
	 We saw in \cref{prop_inclusion} that if \(\Gamma\) is an exact group, then the largest type I ideal in \(C_u^*(\Gamma)\) is contained in the largest AF-ideal. If \(\Gamma\) is finite, the converse inclusion also holds since the uniform Roe algebra itself is both AF and type I in this case. If \(\Gamma\) is an infinite locally finite group, the inclusion is proper since \(C_u^*(\Gamma)\) is AF, but cannot be type I since the simple quotients are unital and infinite-dimensional (by freeness of the action of a group on its Stone--\v{C}ech compactification). We will see that equality does not hold for finitely generated groups in general since it fails even for the integers. We first formulate a geometric property that guarantees the existence of an embedding of the CAR-algebra into a uniform Roe algebra, before providing an example of a subset of the integers of asymptotic dimension zero with said property. 
	
	\begin{proposition}
		\label{prop_CAR}
		Suppose that \(X\) is a uniformly locally finite space admitting a sequence \((\Pi_n)_{n=1}^\infty\) of partitions such that for every \(n \geq 1\), we have 
		\begin{enumerate}
			\item \(\Pi_{n}\) is a refinement of \(\Pi_{n+1}\); 
			\item \(\abs{P} = 2^n\) for every \(P \in \Pi_n\); 
			\item \(\sup \{ \operatorname{diam}(P) : P \in \Pi_n\}< \infty\).  
		\end{enumerate}
		Then \(C_u^*(X)\) contains an embedded copy of the CAR-algebra. 
	\end{proposition}
	
	\begin{proof}
		We consider the standard representation of \(C_u^*(X)\) on \(\ell^2(X)\). Fix a total order on \(X\), and for any finite subset \(P \subseteq X\), let \(f_P \colon P \to \{1, \ldots, \abs{P}\}\) denote the corresponding order-isomorphism, as in the proof of \cref{thm_AF_ideals}. For each \(n \geq 1\), this specifies an isomorphism \(\varphi_P \colon M_{2^n}(\C) \to \B(\ell^2(P))\) and we consider the embedding \[\varphi_n := \prod_{P \in \Pi_n} \varphi_P \colon M_{2^n}(\C) \to \B(\ell^2(X)),\]
		whose image, denote it by \(\A_n\), is contained in \(C_u^*(X)\) by assumption (iii). By assumption (i), we have \(\A_n \subseteq \A_{n+1}\) unitally. This provides an increasing sequence of C*-subalgebras of \(C_u^*(X)\), whose closed union is isomorphic to the CAR-algebra. 
	\end{proof}
	
	\begin{example}
		\label{ex_typeI_notAF}
		We construct a subset \(A \subseteq \Z\) of asymptotic dimension zero satisfying the conditions of \cref{prop_CAR}. By \cref{thm_AF_ideals}, the corresponding ideal \(\I_A\) is AF, and by \cref{prop_CAR}, \(\I_A\) contains the CAR-algebra since \[C_u^*(A) \cong 1_A C_u^*(\Z) 1_A \subseteq \I_A.\]Thus \(\I_A\) is AF but not type I, and we conclude that the largest type I-ideal in \(C_u^*(\Z)\) is properly contained in the largest AF-ideal. 
		
		For this construction, we will use the 2-adic valuation function \(\nu_2 \colon \Z \setminus \{0\} \to \N\) that assigns to an integer \(k\) the highest power \(\nu_2(k)\) such that \(2^{\nu_2(k)}\) divides \(k\). The only property of the 2-adic valuation we will need is that for any \(x, y\), we have \(\nu_2(x+y) \geq \min\{ \nu_2(x), \nu_2(y) \}\) with equality if \(\nu_2(x) \neq \nu_2(y)\).  
		
		Let \(A = \{x_0, x_1, x_2, \ldots \}\) be the subset of the integers obtained inductively as follows: Put \(x_0 = 1\) and let \(x_{n+1} = x_n + \nu_2(2(n+1))\) for \(n \geq 0\). The resulting set starts off as follows: 
		\[A = \{1, 2, 4, 5, 8, 9, 11, 12, 16, 17, \ldots \}.\]
		Inductively, define a sequence \((\Pi_n)_{n=1}^\infty\) of partitions of \(A\) by letting \(\Pi_1\) consist of consecutive pairs. For \(n \geq 1\), \(\Pi_{n+1}\) is obtained by joining neighboring members of \(\Pi_n\). Clearly, this sequence of partitions satisfies the first two conditions. To show that it also satisfies the last condition, we will show that \[\sup \{ \operatorname{diam}(P) : P \in \Pi_n \} \leq (2^n-1)n\] for every \(n \geq 1\). Indeed, consider the \(k\)th interval \(P \in \Pi_n\) (counted from 0). Then \(P\) consists of \(2^n\) elements separated by the values of \(\nu_2\) on some integral interval 
		\[ [2+k2^{n+1}, 2+k2^{n+1} + 2(2^n-2)].\] We claim that for any element \(x\) in that interval, we have \(\nu_2(x) \leq n\). Indeed, \(x\) may be written as \(x = 2(k2^n + r+1) \) for \(0 \leq r \leq 2^n-2\). Moreover, we have 
		\[\nu_2(k2^n) \geq n \quad \text{and} \quad \nu_2(r+1) < n.\] Hence, by the property mentioned above, we have \[\nu_2(x) = \nu_2(k2^n + r+1) + 1 = \min \{ \nu_2(k2^n), \nu_2(r+1) \} + 1 < n+1.\] We conclude that the diameter of \(P\) is upper bounded by \((2^n-1)n\). 
		
		Moreover, notice that since the \(k\)th interval above ends in 
		\[2+k2^{n+1}+2(2^n-2) = k2^{n+1} +2^{n+1}-2 < (k+1) 2^{n+1}\] and the \(k+1\)th interval starts at \(2+(k+1)2^{n+1}\), there is a gap between them of length \(\nu_2(2^{n+1}) = n+1\). Hence, \(\Pi_n\) is in fact \((n+1)\)-separated. This shows that \(A\) has asymptotic dimension zero. 
	\end{example}
	
	\begin{remark}
		Clearly, the number 2 may be replaced by any prime \(p\) in \cref{prop_CAR} and \cref{ex_typeI_notAF} to obtain an embedding of the UHF-algebra of type \(p^\infty\). 
	\end{remark}
	
	\printbibliography{}
	
\end{document}